\documentclass[12pt]{article}
\usepackage{amssymb}
\def\sign{\mathop{\rm Sign}}
\def\ker{\mathop{\rm Ker}}
\makeatletter\@addtoreset {equation}{section}\makeatother

\newtheorem{theorem}{Theorem}
\newtheorem{lemma}{Lemma}

\newtheorem{corollary}{Corollary}

\makeatletter
\def\Ddots{\mathinner{\mkern1mu\raise\p@
\vbox{\kern7\p@\hbox{.}}\mkern2mu
\raise4\p@\hbox{.}\mkern2mu\raise7\p@\hbox{.}\mkern1mu}}
\makeatother

\newenvironment{proof}{
    \noindent {\it Proof.}}{\hfill$\Box$
}

\usepackage[dvips]{epsfig}
\usepackage{graphicx}

\begin{document}

\title{\bf On Factorization of a Perturbation of a J-selfadjoint Operator Arising in Fluid
Dynamics.}

\author{Marina Chugunova \\ {\small Department of Mathematics, University of
Toronto, Canada} \\ Vladimir Strauss \\ {\small Department of Pure
\& Applied Mathematics, Sim\'on Bol\'{\i}var University, Venezuela}
}

\date{\today}
\maketitle

Abstract: We prove that some perturbation of a J-selfadjoint second
order differential operator admits factorization and use this new
representation of the operator to prove compactness of its resolvent
and to find its domain.

AMS classification codes: primary 47B10, 34L40; secondary 35M10.


\section{Introduction}
Depending on the parameters involved, the dynamics of the film of
viscous fluid can be described by different asymptotic equations.
Under the assumption that the film is thin enough for viscous
entrainment to compete with gravity, the time evolution model of a
thin film of liquid on the inner surface of a cylinder rotating in a
gravitational field was based on the lubrication approximation and
examined by Benilov, O'Brien, and Sazonov \cite{Benilov1, Benilov3}.
The related Cauchy problem has the following form:
\begin{equation}
\label{heat PDE}
 y_t + l[y] = 0, \quad y(0,x) = y_0, \quad y(-\pi,t) = y(\pi,t), \quad x\in [-\pi,\pi], \quad t > 0
\end{equation}
where
\begin{equation}
\label{operation}
 l[y] = \frac{d\,}{d \, x} \left((1 - a \, \cos x)y(x) + b \,
\sin x \cdot \frac{d\, y(x)}{d \, x} \right), \quad a, b > 0
\end{equation}

Eigenmode solutions are very important in stability analysis,
because even a single growing mode can destabilize an otherwise
stable system. In case when all modes are bounded in time and the
corresponding eigenfunctions form a complete set, the system
normally regarded as a stable one. Because, an arbitrary initial
condition can be represented as a series of these eigenmodes; and
since all of them are stable, so expected to be the solution to the
initial- value problem.

There are however counterexamples to the arguments above when each
term of the series is bounded but the series as a whole diverges and
the solution develops a singularity in a finite time. This effect
was observed by Benilov, O'Brien, and Sazonov \cite{Benilov1} for
the problem \ref{heat PDE} when parameter in (\ref{operation}) $\,\,
a = 0$. For this case when the effect of gravitational drainage was
neglected because of infinitesimally thin film they studied
stability of the problem asymptotically and numerically. It was
shown that even for infinitely smooth initial values numerical
solutions blow up after a small number of iterations.

The spectrum of the linear operator $L$ that is defined by the
operation $l[.]$ and periodic boundary conditions $y(-\pi) = y(\pi)$
for the special case when the parameter $a=0$ was studied rigorously
in \cite{Davies, ChugPel, Weir}. Using different approaches they
justified that if the parameter $b$ restricted to the interval
$[0,2]$ then the operator $L$ is well defined in the sense that it
admits closure in $L^2 (-\pi, \pi)$ with non-empty resolvent set
without breaking the boundary conditions $y(-\pi) = y(\pi)$. The
spectrum of the operator $L$ is discrete and consists of simple pure
imaginary eigenvalues only. As a result all eigenfunctions have the
following symmetry $y_\lambda(-x) = \overline{y_\lambda(x)}$. The
more general operator with the function $sin(x)$ replaced by the
arbitrary $2\pi$-periodic functions was studied in \cite{PT-sym} and
it was proved that this operator multiplied by $i$ belongs to a wide
class of $PT$-symmetric operators which are not similar to
self-adjoint but nevertheless possesses purely real spectrum due to
some obvious and hidden symmetries.

The phenomenon of the coexistence of the neutrally stable modes with
explosive instability of the numerical solutions \cite{Benilov1}
(which correspond to drops of fluid forming on the ceiling of the
cylinder where the effect of the gravity is the strongest) was
studied analytically and explained in terms of the absence of the
Riesz basis property of the set of eigenfunctions in
\cite{ChKP08Preprint}. The question of a conditional basis property
of the set of eigenfunction is still open.

For the case when $a \neq 0$, as it was discussed in
\cite{Benilov3}, the spectral properties of the operator $L$ are not
expected to differ a lot from the case $a = 0$.

The goal of this paper is to find the factorization of the operator
$L$ (under some restrictions on parameters $a$ and $b$) that would
be in some sense similar to one we constructed for the special case
$a = 0$ in \cite{ChugStr} (in this case the operator $L$ is
$J$-self-adjoint with the operator $J$ defined as a shift $J(f(x)) =
f(\pi - x)$)  and to examine some properties of the operator $L$
using this factorization. The main difficulty to overcome here is an
existing coupling between two subspaces spanned by positive and
negative Fourier exponents which are not invariant subspaces of the
operator $L$ if $a \neq 0$. We also prove that the non-self-adjoint
differential operator $L$ has compact resolvent and as result
spectrum of $L$ is discrete with the only accumulating point at
infinity.

\section{Factorization of the non-self-adjoint operator $L$.}

We denote by $\mathfrak{D}(T)$ and $\mathfrak{R}(T)$ the domain and
the range of linear operator $T$ respectively. The notation
$\mathcal{L}^2$ is used for the standard Lebesgue space of scalar
functions defined on the interval $(-\pi , \pi)$. From here on $L$
is the indefinite convection-diffusion operator $L:$
$$ \mathcal{L}^2\mapsto \mathcal{L}^2, \quad (Ly)(x) =
\frac{d\,}{d \, x} \left((1 - a \, \cos x)y(x) + b \, \sin x \cdot
\frac{d\, y(x)}{d \, x} \right)
$$ with the domain of all
absolutely continuous $2\pi$-periodic functions $y(x)$ such that
$(Ly)(x)\in \mathcal{L}^2$.

In addition, we define the operator $S:$
$$\mathcal{L}^2\mapsto \mathcal{L}^2, \quad (Sy)(x)=y^{\prime}(x),$$
where $y^{\prime}(x)\in \mathcal{L}^2$, $y(-\pi )=y(\pi )$, and the
operator $M:$ $\mathcal{L}^2\mapsto \mathcal{L}^2$, $$ (My)(x)\colon
=(1 - a \, \cos x)y(x) + b \, \sin x \cdot \frac{d\, y(x)}{d \, x}
$$ with the domain of all functions $y(x)\in \mathcal{L}^2$
absolutely continuous on $(-\pi,\,0)\cup (0,\, \pi )$ and such that
$(My)(x)\in \mathcal{L}^2$. Note, that, for example,
$y(x)=x^{-1/3}\in \mathfrak{D}(M)$. The operator $M$ can also be
represented by the following expression
$$(My)(x)=\big(1 - (a+b) \, \cos x)y(x) +
b \, (\sin x \cdot  y(x)\big)^{\prime}. $$

\begin{theorem}\label{main}If the parameters $a$ and $b$ satisfy the inequality
 $2 a + b < 2$, then $L$ is a closed operator with a closed range and $L=SM$.\end{theorem}
\begin{proof}
 Let us consider the operator $A:$ $$\mathcal{L}^2\mapsto
\mathcal{L}^2, \quad (Ay)(x)=(\sin(x)y(x))^{\prime}$$ with
$\mathfrak{D}(A)=\{y(x)\ |\ y(x), (Ay)(x)\in \mathcal{L}^2\}$. Then
a function $y(x)$ can be written as
\begin{equation}\label{1} y(x)=\frac{1}{\sin(x)}\cdot \big(
c+\int_0^x \theta(t)dt\big) ,\quad \theta(t)\in
\mathcal{L}^2.\end{equation} If $x>0$ then
$$ |\int_0^x
\theta(t)dt|\leq \frac{1}{\sin(x)}\cdot \alpha(x)\cdot x^{1/2},$$
where $\alpha(x)=\big(\int_0^x |\theta(x)|^2dx\big)^{1/2}$. Since
the two summands in (\ref{1}) have different orders of growth as
$x\to 0$ this implies that if $y(x)\in \mathcal{L}^2$ then $c=0$ and
\begin{equation}\label{2} y(x)=\frac{1}{\sin(x)}\cdot
\int_0^x \theta(t)dt .\end{equation} Moreover,
\begin{equation}\label{3} |y(x)|\leq\frac{x^{1/2}}{\sin(x)}\cdot
\alpha(x) .\end{equation} A small modification of the same reasoning
leads to the following estimation for every $x\in (-\pi ,\pi)$
\begin{equation}\label{4} |y(x)|\leq\frac{|x|^{1/2}}{|\sin(x)|}\cdot
\alpha(x) .\end{equation} with $\alpha(x)=\big|\int_0^x
|\theta(x)|^2dx\big|^{1/2}$.

Alternatively the same function $y(x)$ can be written as
\begin{equation}\label{5} y(x)=\frac{1}{\sin(x)}\cdot \big(
\tilde{c}-\int_x^{\pi} \theta(t)dt\big) ,\quad \theta(t)\in
\mathcal{L}^2.\end{equation} with the same  $\theta(x)$ as in
(\ref{1}). Representation (\ref{5}) yields the following relations
\begin{equation}\label{6} y(x)=\frac{-1}{\sin(x)}\cdot
\int_x^{\pi} \theta(t)dt \end{equation} and
\begin{equation}\label{7} |y(x)|\leq\frac{(\pi -x)^{1/2}}{\sin(x)}\cdot
\beta(x) \end{equation} with $\beta(x)=\big|\int_x^{\pi}
|\theta(x)|^2dx\big|^{1/2}$.

It follows from (\ref{2}) and (\ref{6}) that
\begin{equation}\label{8} \int_{0}^{\pi}\theta(t)dt=0.\end{equation}

Starting from the point $-\pi$ one can also obtain that
\begin{equation}\label{9} y(x)=\frac{1}{\sin(x)}\cdot \int_{-\pi}^x
\theta(t)dt,\end{equation}
\begin{equation}\label{10} |y(x)|\leq\frac{(-\pi +x)^{1/2}}{|\sin(x)|}\cdot
\gamma(x) \end{equation} with $\gamma(x)=\big(\int_{-\pi}^x
|\theta(x)|^2dx\big|^{1/2}$ and
\begin{equation}\label{11} \int_{-\pi}^{0}\theta(t)dt=0.\end{equation}
\par
Now we are ready to calculate $M^{*}$. Using smooth functions $y(x)$
such that $y(x)\equiv 0$ in some neighborhoods of the points $-\pi$,
$0$ and $\pi$ (neighborhoods depend of $y(x)$) it easy to show that
$$(M^*z)(x)=\big(1 - a \, \cos x)z(x) -
b \, (\sin x \cdot  z(x)\big)^{\prime} $$ for every
$z(x)\in\mathfrak{D} (M^*)$. Since the condition
$z(x)\in\mathfrak{D} (M^*)$ yields \begin{equation}\label{12}
z(x)\in \mathcal{L}^2\ \mbox{ and }\ (M^*z)(x)\in \mathcal{L}^2\,
,\end{equation} for $z(x)$ are fulfilled the conditions of the type
(\ref{4}), (\ref{7}) and (\ref{10}). Taking into account the latter
one can check that
$$(My,z)=(y,M^*z)$$ for every $y(x)\in\mathfrak{D} (M)$ and $z(x)$
under Conditions (\ref{12}). Thus,
$\mathfrak{D}(M)=\mathfrak{D}(M^*)$. The same reasoning shows that
$M^{**}=M$, so $M$ is closed.
\par
Let
$$
(1 - a \, \cos x)y(x) + b \, \sin x \cdot \frac{d\, y(x)}{d \, x}
=u(x), \quad y(x),u(x)\in\mathcal{L}^2.
$$ Our aim is to express $y(x)$ via $u(x)$. Let $x\in(-\pi,\,\pi),\, x\not=0$. Then $y(x)=$
$$\big(c(1+\sign x)/2+c_1(1-\sign x)/2\big)\cdot (\sin |x|)^{a/b}\cdot (\cot |x|/2)^{1/b} \,
+$$ $$ \frac{1}{b}(\sin |x|)^{a/b}\cdot (\cot |x|/2)^{1/b}\int_0^x
u(t)(\sin |t|)^{-\frac{a}{b}}\cdot (\sin t)^{-1}\cdot  (\tan
|t|/2)^{1/b}dt,$$ where $c$ and $c_1$ are constants. The estimations
that follow closely depend of a relation between $a$ and $b$. We
assume that $2a+b<2$. Then for $x>0$
$$\int_0^x
u(t)(\sin t)^{-(\frac{a}{b}+1)}\cdot (\tan t/2)^{1/b}dt=\int_0^x
v(t)(t)^{-(\frac{a}{b}+1)}\cdot (t/)^{1/b}dt,$$ where
$v(t)=u(t)(\sin t)^{-(\frac{a}{b}+1)}\cdot (\tan t/2)^{1/b} (
t)^{(\frac{a}{b}+1)}\cdot ( t)^{-1/b}$, so,
$$ |\int_0^x
u(t)(\sin t)^{-(\frac{a}{b}+1)}\cdot (\tan t/2)^{1/b}dt|\leq $$
$$\sqrt{\frac{b}{2-2a-b}}\cdot x^{\frac{2-2a-b}{2b}}\cdot \big(\int_0^x |v(t)|^2dt\big)^{1/2}.$$
Thus, the first summand (if $c\not=0$) for $y(x)$ has the order
$x^{\frac{a-1}{b}}$ and the second one has the order
$x^{-1/2}\alpha(x)$ with $\lim\limits_{x\to 0}\alpha(x)=0$. Since
$u(x)\in\mathcal{L}^2$, $c=0$. The same reasoning shows that
$c_1=0$. Thus, \begin{equation}\label{13} y(x)=\frac{1}{b}(\sin
|x|)^{a/b}\cdot (\cot |x|/2)^{1/b}\int_0^x u(t)(\sin
|t|)^{-\frac{a}{b}}\cdot (\sin t)^{-1}\cdot(\tan
|t|/2)^{1/b}dt.\end{equation} In particular, for $u(x)\equiv 1$ we
have
$$y_0(x)\colon =\frac{1}{b}(\sin
|x|)^{a/b}\cdot (\cot |x|/2)^{1/b}\int_0^x (\sin
|t|)^{-\frac{a}{b}}\cdot (\sin t)^{-1}\cdot(\tan |t|/2)^{1/b}dt.$$
Some elementary estimations show that there are finite limits
$\lim\limits_{x\to 0}y_0(x)$, $\lim\limits_{x\to -\pi}y_0(x)$ and
$\lim\limits_{x\to \pi}y_0(x)$ with $\lim\limits_{x\to
-\pi}y_0(x)=\lim\limits_{x\to \pi}y_0(x)$. Let us show these
relations. First, for $t>0$ we define $w(t)\colon =\big(\frac{\sin
t}{t}\big)^{-\frac{a}{b}-1}\cdot\big(\frac{\tan
(t/2)}{t}\big)^{1/b}$. Then $\lim\limits_{t\to +0}w(t)=(1/2)^{1/b}$.
Moreover, for $x>0$ $$y_0(x)\colon =\frac{1}{b}(\sin x)^{a/b}\cdot
(\cot x/2)^{1/b}\int_0^x w(t)t^{-\frac{a}{b}-1+1/b}dt,$$ so
$$y_0(x) =\frac{1}{b}(\sin x)^{a/b}\cdot
(\cot x/2)^{1/b}\frac{b}{1-a} w(\xi_x)(x)^{-\frac{a}{b}+1/b},$$
where $\xi_x\in (0,x)$. The latter yields $y(0)\colon
=\lim\limits_{x\to +0}y(x)=\frac{1}{1-a}$. Second, for $t<\pi$ we
define $w_+(t)\colon =\big(\frac{\sin t}{\pi
-t}\big)^{-\frac{a}{b}-1}\cdot\big(\frac{(\pi -t)}{2}\cdot \tan
(t/2)\big)^{1/b}$. Then $\lim\limits_{t\to +\pi-0}w_+(t)=1$ and for
$z(x)\colon =(1+a)\cdot y_0(x)\cdot
\big(\sin(x/2)\big)^{\frac{2}{b}}$ we have
$$z(x) =\frac{1+a}{b}(\sin x)^{\frac{a+1}{b}}\cdot
\int_0^x w_+(t)(\pi -t)^{-\frac{a}{b}-1-1/b}dt.$$ Let us fix
$\epsilon >0$ Then there is $\delta >0$ such that $1-\epsilon <
w_+(x)<1+\epsilon$ for every $x\in (\pi -\delta, \pi)$. Next, for
the same $x$
$$ z(x)= \frac{1+a}{b}(\sin x)^{\frac{a+1}{b}}\cdot \big(
\int_0^{\pi -\delta} w_+(t)(\pi -t)^{-\frac{a}{b}-1-1/b}dt+$$ $$
\int_{\pi-\delta}^x w_+(t)(\pi -t)^{-\frac{a}{b}-1-1/b}dt\big).$$
Since
$$\int_{\pi-\delta}^x w_+(t)(\pi -t)^{-\frac{a}{b}-1-1/b}dt\big)=\frac{b}{1+a} w(\nu_{x,\delta})\big( (\pi-x)^{-\frac{1+a}{b}}-\delta^{-\frac{1+a}{b}} \big)$$
with $\nu_{x,\delta}\in (\pi -\delta, \pi) $,
$$\frac{b}{1+a} (1-\epsilon)\big( (\pi-x)^{-\frac{1+a}{b}}-\delta^{-\frac{1+a}{b}} \big)\leq\int_{\pi-\delta}^x w_+(t)(\pi -t)^{-\frac{a}{b}-1-1/b}dt\big)\leq$$ $$\frac{b}{1+a} (1+\epsilon)\big( (\pi-x)^{-\frac{1+a}{b}}-\delta^{-\frac{1+a}{b}} \big).$$ Moreover, for fixed $\delta$ $$\lim\limits_{t\to +\pi-0}
\frac{1+a}{b}(\sin x)^{\frac{a+1}{b}}\cdot \int_0^{\pi -\delta}
w_+(t)(\pi -t)^{-\frac{a}{b}-1-1/b}dt=0$$ and $$\lim\limits_{t\to
+\pi-0} (\sin x)^{\frac{a+1}{b}}\cdot  w(\nu_{x,\delta})
\delta^{-\frac{1+a}{b}}=0,$$ so the equality $\lim\limits_{x\to
\pi}y_0(x)=\frac{1}{1+a}$ is practically evident. Note also that
$y_0(x)$ is even.

Thus, the function $y_{0}(x)$ is continuous on $[-\pi,\pi]$ and
satisfies the periodic conditions.
\par
Now let $u(x)=c+\int_0^x\phi(t)dt$, where $c=const$ and $\phi(x)\in
\mathcal{L}^2$. Then for $x>0$: $|\int_0^x\phi(t)dt|\leq
x^{1/2}\big(\int_0^x|\phi(t)|^2\big)^{1/2}$. The latter estimation
and (\ref{13}) yield $\lim\limits_{x\to 0}y(x)=c\cdot y_0(0)$. The
same function can be re-written as following
$u(x)=c_++\int_{\pi}^x\phi(t)dt$ or
$u(x)=c_-+\int_{-\pi}^x\phi(t)dt$. Then the estimations
$|\int_{\pi}^x\phi(t)dt|\leq (\pi-
x)^{1/2}\big(-\int_{\pi}^x|\phi(t)|^2\big)^{1/2}$, $x\in (o,\pi)$
and $|\int_{-\pi}^x\phi(t)dt|\leq (\pi+
x)^{1/2}\big(\int_{-\pi}^x|\phi(t)|^2\big)^{1/2}$, $x\in (-\pi,o)$
together with Representation (\ref{13}) yield  $\lim\limits_{x\to
\pi}y(x)=c_+\cdot y_0(\pi)$ and $\lim\limits_{x\to
-\pi}y(x)=c_-\cdot y_0(-\pi)$. Thus, $y(x)$ satisfies the periodic
condition if and only $u(x)$ satisfies. The latter yields the
equality
\begin{equation}\label{14}L=SM\, .\end{equation}
 Moreover, we shown that for
every $u(x)=c+\int_0^x\phi(t)dt$ with $\phi(x)\in \mathcal{L}^2$
there is absolutely continuos on $[-\pi, \pi ]$ function $y(x)$,
such that $(My)(x)=u(x)$, so $\mathfrak{R}(M)$ is dense in
$\mathcal{L}^2$ or, equivalently, $\ker(M^*)=\{ 0\}$.
\par
Note, that $M$ is boundedly invertible. Indeed, $M=D+iC$, where
$C\colon$ $\mathcal{L}^2\mapsto \mathcal{L}^2$,
$$(Cy)(x)\colon =i\Big\{ \frac{b}{2} \, \cos x\, y(x) -
b \, (\sin x \cdot y(x))^{\prime} \Big\}
$$
and $D\colon$ $\mathcal{L}^2\mapsto \mathcal{L}^2$,
$$(Dy)(x)\colon =\big(1 - (a+\frac{b}{2}) \, \cos x\big)y(x)
\,.
$$
\noindent It easy to check that $C$ is self-adjoint. Moreover, $D$
is positive, bounded and boundedly invertible, so the problem of
invertibility of $M$ is equivalent  to the problem of regularity of
non-real numbers for a self-adjoint operator (for a more detail
reasoning see, for instance, \cite{ChugStr}).
\par
Now let us prove that $L$ is closed. The operator $S$ restricted on
the subspace $\mathcal{L}_1\subset\mathcal{L}^2$ of functions
orthogonal to constants has a bounded inverse. Let us find
$M^{-1}(\mathcal{L}_1)$. If $(My)(x)\in\mathcal{L}_1$, then
$\int_{-pi}^{\pi}(My)(x)d\, x=0$, but $\int_{-\pi}^{\pi}(y(x)\sin
x)^{\prime}d\, x=0 $, so $y(x)\in M^{-1}(\mathcal{L}_1)$ if and only
if $y(x)\in\mathfrak{D}(M)$ and $\int_{-\pi}^{\pi}\big(1 - (a+b) \,
\cos x\big)y(x)d\, x=0$. Let $\mathcal{L}_2\colon =\Big\{\big(1 -
(a+b) \, \cos x\big)\Big\}^{\perp}$. Since for $y_0(x)$ we have
$2\pi=\int_{-\pi}^{\pi}(My_0)(x)d\, x=\int_{-\pi}^{\pi}y_0(x)\big(1
- (a+b) \, \cos x\big)d\, x$, $y_0(x)\not\in\mathcal{L}_2$ and
$\mathcal{L}^2=\mathcal{L}_2\dot{+}\{\mu\cdot
y_{0}(x)\}_{\mu\in\mathbb{C}}$.
Now let a sequence $\{y_k(x)\}$ be such that $y_k(x)\to y(x)$ and
$z_k(x)=(Ly_k)(x)\to z(x)$. Then $(My_k)(x)=v_k(x)+c_k$, where
$v_k(x)=((S|_{\mathcal{L}_1})^{-1}z)(x)$, $c_k=const$,
$k=1,2,\ldots$ . Since $(S|_{\mathcal{L}_1})^{-1}$ is bounded, the
sequence $\{v_k(x)\}$ has a limite $v(x)$. In turn, in virtue of
similar reasons the sequence $w_k(x)=(M^{-1}v_k)(x)$ also has a
limite. Simultaneously
$y_k(x)=(M^{-1}(v_k+C_k))(x)=w_k(x)+c_ky_0(x)$. Thus, the sequence
$\{c_k\}$ has a limite. The rest is straightforward.
\end{proof}
\begin{corollary} If the parameters $a$ and $b$ satisfy the inequality
 $2 a + b < 2$, then the set $\mathfrak{D}(L)$ is the
linear sub-manifold  $H$ of the Sobolev space $H^1 (-\pi, \pi)$:
$$ \mathfrak{D}(L) = H : { f \in H^1 (-\pi, \pi),\quad f(\pi) = f(-\pi,) \quad
\sin(x)f' \in H^1 (-\pi, \pi) }$$ and is a Hilbert space with the
norm defined as:
$$||f||^2 = || f||^2_{H^1} + || \sin(x)f'(x)||^2_{H^1}.$$
\end{corollary}
The reasoning of this corollary is the same as the reasoning of the
corresponding proposition in \cite{ChugStr}.
\par
For the next step we need the following simple remark.
\begin{lemma} Let $\mathcal{H}$ be a Hilbert space, $x_1\,
,x_2\in\mathcal{H}$, $(x_1,x_2)\not=0$. Let $\mathcal{H}_1\colon =
\{x_1\}^{\perp}$, $\mathcal{H}_2\colon =\{x_2\}^{\perp}$. Let $P_1$
and $P_2$ be ortho-projections onto the subspaces $\mathcal{H}_1$
and $\mathcal{H}_2$ respectively. Then $P_2|_{\mathcal{H}_1}$ is
one-to-one mapping onto $\mathcal{H}_2$.\label{trivial}
\end{lemma}
\begin{proof} Let $\mathcal{H}_3\colon =\mathcal{H}_1\cap
\mathcal{H}_2$. With no loss of generality we can assume that
$\|x_1\|=\|x_2\|=1$, $(x_1,x_2)=\alpha >0$. Then $\mathcal{H}_1 =
\{\mu\cdot (x_2-\alpha\cdot x_1)\}_{\mu\in \mathbb{C}}\oplus
\mathcal{H}_3$, $\mathcal{H}_2 = \{\mu\cdot (x_1-\alpha\cdot
x_2)\}_{\mu\in \mathbb{C}}\oplus \mathcal{H}_3$. Since
$P_2|_{\mathcal{H}_3}=I_{\mathcal{H}_3}$ and $P_2(x_2-\alpha\cdot
x_1)=\alpha\cdot (x_1-\alpha\cdot x_2)$, the rest is evident.
\end{proof}
\begin{theorem}\label{main1}If the parameters $a$ and $b$ satisfy the inequality
 $2 a + b < 2$, then  the resolvent of $L$ is compact of the Hilbert-Schmidt type.\end{theorem}
\begin{proof} Let $\mathcal{L}_0\subset\mathcal{L}^2$ be the subspace of constants,
$\mathcal{L}_1\colon =\mathcal{L}_0^{\perp}$. Since
$\mathfrak{R}(L)=\mathcal{L}_1$, the operator $L$ has the following
matrix representation $$ L= \left[ \begin{array}{cc} 0 & 0 \\
L_{10} & L_{11} \end{array} \right]$$ with respect to the
decomposition $\mathcal{L}^2=\mathcal{L}_0\oplus\mathcal{L}_1 $. The
operator $L_{10}$: $1\to a\cdot \sin x$ is bounded, so we need to
analyze the properties of $L_{11}$. From Theorem 1 we have
$L_{11}=SM|_{\mathcal{L}_1}$. Let
$\mathcal{L}_2=M(\mathfrak{D}(M)\cap\mathcal{L}_1)$,
$y_0(x)=(M^{-1}1)(x)$, $z_0(x)=((M^*)^{-1}1)(x)$. Then for
$y(x)\in\mathfrak{D}(M)\cap\mathcal{L}_1$ we have
$0=(y,1)=(y,M^*z_0)=(My,z_0)$, so $\mathcal{L}_2=\{ z_0\}^{\perp}$.
From the other hand, $(1,z_0)=(My_0,z_0)=(y_0,M^*z_0)=(y_0,1)$.
Since (see the proof of Theorem \ref{main}) $(y_0,1)\not= 0$, the
pair $\{ 1, z_0\}$ is under the conditions of Lemma \ref{trivial}.
Let $P_1$ be the ortho-projection onto $\mathcal{L}_1$. Then
$L_{11}=S\cdot (P_1|_{\mathcal{L}_2})\cdot (M|_{\mathcal{L}_1})$, so
$L_{11}^{-1}=
(M^{-1}|_{\mathcal{L}_2})\cdot(P_1|_{\mathcal{L}_1})^{-1}\cdot
(S|_{\mathcal{L}_1})^{-1}$. Thus, $L_{11}^{-1} $ is an operator of
the Hilbert-Schmidt type. Since
$$R_{\lambda}(L)= \left[ \begin{array}{cc} R_{\lambda}(0) & 0
\\\frac{1}{\lambda}R_{\lambda}(L_{11})
L_{10} & R_{\lambda}(L_{11}) \end{array} \right]\, , $$ the rest is
straightforward.
\end{proof}

{\bf Acknowledgement.}  M.C. is supported by the NSERC Postdoctoral
Fellowship.

\begin{tabular}{ll}

     Department  & \qquad Departamento de  Matem\'aticas \\
     of Mathematics &\qquad Puras y Aplicadas\\
     University of Toronto & \qquad  Universidad Sim\'on Bol\'{\i}var\\
     Toronto, Ontario M5S 2E4 & \qquad  Caracas 1080\\
     Canada & \qquad  Venezuela\\
     e-mail: chugunom@math.toronto.edu & \qquad    e-mail: str@usb.ve\\

\end{tabular}

\end{document}